\documentclass[11pt,a4paper]{amsart}
\usepackage{mypackage}

\date{\today}

\author{Bertrand Deroin \and Nicolas Tholozan}
\thanks{B.D.'s research was partially supported by ANR-13-BS01-0002.}
\address{CNRS \\ UMR 8553 \\ D\'epartement de Math\'ematiques et Applications  \\ \'Ecole Normale Sup\'erieure \\ 45, rue d'Ulm \\ 75005 Paris, France.}
\email{bertrand.deroin@ens.fr}
\address{University of Luxembourg, Campus Kirchberg\\
Mathematics Research Unit, BLG\\
6, rue Richard Coudenhove-Kalergi\\
L-1359 Luxembourg\\}
\email{nicolas.tholozan@uni.lu}

\title[Super-maximal representations]{Super-maximal representations from fundamental groups of punctured spheres to $\text{PSL} (2,\mathbb R)$}

\begin{document}

\maketitle

\begin{abstract}
We study a particular class of representations from the fundamental groups of  punctured spheres $\Sigma_{0,n}$  to the group $\text{PSL} (2,\mathbb R)$ (and their moduli spaces), that we call \textit{super-maximal}. Super-maximal representations are shown to be \textit{totally non hyperbolic}, in the sense that every simple closed curve is mapped to a non hyperbolic element. They are also shown to be \textit{geometrizable} (appart from the reducible super-maximal ones) in the following very strong sense : for any element of the Teichm\"uller space $\mathcal T_{0,n}$, there is a unique holomorphic equivariant map with values in the lower half-plane $\mathbb H^-$. In the relative character variety, the components of super-maximal representations are shown to be compact, and symplectomorphic (with respect to the Atiyah-Bott-Goldman symplectic structure) to the complex projective space of dimension $n-3$ equipped with a certain multiple of the Fubiny-Study form that we compute explicitly (this generalizes a result of Benedetto--Goldman \cite{BenedettoGoldman} for the sphere minus four points). Those are the unique compact components in relative character varieties of $\text{PSL}(2,\mathbb R)$. This latter fact will be proved in a companion paper.
\end{abstract}

\section*{Introduction}

\subsection{Overview}  Let $\Sigma_{g,n}$ be a surface obtained from a connected oriented closed surface of genus $g$ by removing $n$ points, called the punctures. We assume in the sequel that the Euler characteristic of $\Sigma_{g,n}$ is negative. Let also $G= \text{PSL}(2,\mathbb R)$ be the group of isometries of the half-planes $\mathbb H^\pm = \{ z \in \mathbb C \ |\ \pm \Im z>0 \}$ equipped with the metrics $\frac{dx^2+ dy^2}{y^2}$ of curvature $-1$, where $z= x+ iy$.  We denote by $\text{Hom}(\pi_1(\Sigma_{g,n}), G)$ the set of representations from the fundamental group of $\Sigma_{g,n}$ to $G$, and by $\text{Rep} (\pi_1(\Sigma_{g,n}), G) = \text{Hom}(\pi_1(\Sigma_{g,n}), G) / G$ its quotient by the action of $G$ by conjugation. We will call this latter the character variety, even though it is not the algebraic quotient (in the sense of geometric invariant theory) that we are considering here. 

A representation $\rho \in \text{Hom} (\pi_1(\Sigma_{g,n}) ,G ) $ determines a flat oriented  $\ProjR{1}$-bundle over $\Sigma_{g,n}$ which, if we forget the flat connection, is encoded up to isomorphism by a class in $H^2(\Sigma_{g,n},\mathbb Z )$, called the Euler class, and denoted $\textbf{eu}(\rho)$.  In the closed case, i.e. when $n=0$, we have $H^2(\Sigma_{g,0}, \mathbb Z ) \simeq \mathbb Z$, so the Euler class is an integer, and it satisfies the well-known Milnor-Wood inequality :
\begin{equation} \label{eq:MilnorWoodClosedCase} |\textbf{eu} (\rho)| \leq |\chi (\Sigma_{g,0}) |~,
\end{equation} 

as proved by Wood \cite{Wood}, following earlier work of Milnor \cite{Milnor}. All the integral values in the interval \eqref{eq:MilnorWoodClosedCase} are achieved on $\text{Hom} (\Sigma_{g,0}, G)$. Goldman proved that the level sets of the Euler class are connected \cite{Goldman}, and Hitchin that they are indeed diffeomorphic to a vector bundle over some symmetric powers of $\Sigma_{g,0}$ \cite{Hitchin}. Goldman proved in his doctoral dissertation that the Euler class is extremal exactly when the representation is the holonomy of a hyperbolic structure on $\Sigma_{g,0}$ \cite{GoldmanThese}. He conjectured more generally that the components of non zero Euler class are generically made of holonomies of \emph{branched} $\mathbb H^\pm$-structures on $\Sigma_{g,0}$ with $k= |\chi(\Sigma_{g,0})|-|\textbf{eu}|$ branch points (see \cite{Goldmanconjecture} and also \cite{Tan} where the problem is discussed).


This paper is the first in a series aiming at studying the analog picture on the \textit{relative} character varieties when the surface $\Sigma_{g,n}$ is not closed, namely when $n>0$. 
We focus here on a particular family of components of the relative character varieties, that we call \textit{super-maximal}. They occur only on punctured spheres $\Sigma_{0,n}$ for $n\geq 3$, for some particular choices of elliptic/parabolic peripheral conjugacy classes. 

We prove that these components are compact, and more precisely that they are symplectomorphic (with respect to the Atiyah-Bott-Goldman symplectic structure) to the complex projective space of dimension $n-3$, equipped with a certain multiple of the Fubini-Study form that we compute explicitly.  This generalizes to any $n\geq 4$ a result obtained by Benedetto-Goldman in the case $n = 4$ \cite{BenedettoGoldman}. 

We also prove that the \emph{super-maximal representations} (i.e. those lying in super-maximal components) have very special algebraic and geometric properties. First, we prove that they are totally non hyperbolic, namely that no simple closed curve of $\Sigma_{0,n}$ is mapped to a hyperbolic conjugacy class of $G$. Moreover, we prove that they are geometrizable by $\mathbb H^-$-conifolds in a very strong way.

\subsection{Volume,  relative Euler class and the refined Milnor-Wood inequality}


In the closed case,  the Euler class relates intimately to the \textit{volume} of the representation, classically defined by the integral 
\begin{equation}\label{eq:Volume} \text{Vol} (\rho ) = \int _{\Sigma_{g,0} } f^* \big( \frac{dx\wedge dy} {y^2} ) \end{equation}
where $f: \widetilde{\Sigma_{g,0}} \rightarrow \mathbb H^+$ is any $\rho$-equivariant smooth map. Namely, we have $\text{Vol} (\rho) = 2\pi\, \textbf{eu} (\rho)$. Burger and Iozzi \cite{BurgerIozzi} and independently  Koziarz and Maubon \cite{KM}, have extended the definition of  the volume of a representation $\rho : \pi_1(\Sigma_{g,n}) \rightarrow G$ to the case of punctured surfaces. (See also Burger--Iozzi--Wienhard \cite{BIW} for an anologous notion in the case of representations in a Lie group of Hermitian type.) This volume can be defined as a bounded cohomology class, or more trivially as an integral of the form \eqref{eq:Volume}, where the behaviour of the equivariant map is constrained in the neighborhood of the cusp: namely, the completion of the metric $f^* \big( \frac{dx^2 + dy^2}{y^2} \big)$ at the neighborhood of the cusps is assumed to be a cone, a parabolic cusp, or an annulus with totally geodesic boundary. 

The analogous Milnor-Wood inequality
\begin{equation} \label{eq:MilnorWoodVolume}  |\text{Vol} (\rho) | \leq 2\pi |\chi (\Sigma_{g,n})|~, \end{equation}
holds in this context \cite{BIW,KM}. It is also proved in \cite{BIW} that the volume is continuous as a function on $\text{Rep} (\pi_1(\Sigma_{g,n}), G)$ and achieves every value in the interval defined by \eqref{eq:MilnorWoodVolume}. 

The volume heavily  depends on the conjugacy class of the peripherals $\rho (c_i)$, where the $c_i$ are elements of $\pi_1(\Sigma_{g,n})$ freely homotopic to positive loops around the punctures. For instance, its reduction modulo $2\pi$ equals the sum $-\sum_i R(\rho(c_i))$, where $R(\rho(c_i))$ is the rotation number of $\rho(c_i)$ \cite[Theorem 12]{BIW}. In order to understand better the dependance of the volume on the $\rho(c_i)$, it is convenient to introduce the following function:
\[\theta : G \to \R_+ \]
that maps an element $g \in G$ to
\begin{itemize}
\item $0$ if $g$ is hyperbolic or positive parabolic (i.e. a parabolic that translates anti-clockwise the horocycles),
\item $2\pi$ if $g$ is negative parabolic or the identity,
\item the value between $0$ of $2\pi$ of the rotation angle of $g$ when $g$ is elliptic.
\end{itemize}

We will also denote $\theta_i(\rho) = \theta(\rho(c_i))$ and $\Theta(\rho) = \sum_{i=1}^n \theta_i(\rho)$.

\begin{defi}
We define the \emph{relative Euler class} of the representation $\rho$ by
\begin{equation}\label{eq: euler class} 
\euler(\rho)  =   \frac{1}{2\pi} \big( \text{Vol} (\rho) + \Theta(\rho) \big)~. \end{equation}
\end{defi}
By \cite[Theorem 12]{BIW}, the relative Euler class is an integer. In fact, it can be shown that it is the genuine Euler class of the flat $\ProjR{1}$-bundle with monodromy $\rho$, relative to some explicit trivializations above the curves $c_i$. When $\rho$ is geometrizable by a $\H^+$-structure, we have $\euler(\rho) = |\chi(\Sigma)| - k$, where $k$ denotes the number of branched points counted with multiplicity (including those occuring at the cusps). 

\begin{MonThm}\label{t:MilnorWood}
For every representation $\rho : \pi_1 (\Sigma_{g,n}) \rightarrow G$, we have the inequality
\begin{equation} \label{eq: MilnorWoodrelative} \inf \big( -|\chi(\Sigma_{g,n})|  + l , \lceil \frac{1}{2\pi} \Theta(\rho) \rceil \big) \leq \euler(\rho) \leq \sup \big( |\chi(\Sigma_{g,n}) | , \lfloor \frac{1}{2\pi} \Theta(\rho) \rfloor \big) , \end{equation}
where $l$ is the number of elliptic/parabolic/identity conjugacy classes among the $O_i$'s, where identity classes are counted twice.
\end{MonThm} 

In particular, the relative Euler class satisfies the classical Milnor-Wood inequality $|\textbf{eu} (\rho)|\leq |\chi(\Sigma_{g,n})|$ unless $g=0$, in which case it can take no more than two additional values: $\euler(\rho) = n-1$ or $n$. 

Note that the relative Euler class does not distinguish connected components of the character variety (which is connected for surfaces with punctures), mainly because the function $\theta: G \to \R_+$ is only upper semi-continuous while the volume is a continuous function, see \cite[Theorem 1]{BIW}. However, it does distinguish connected components of \emph{relative} character varieties. In a companion paper, we will prove that inequality \eqref{eq: MilnorWoodrelative} is sharp in the relative character varieties, at least when peripherals are elliptic. We will also show that the Euler class  distinguishes their components, appart from the case of the pair of pants or the one-holed torus. This extends Goldman's theorem \cite{Goldman} to the relative case. 

\subsection{Super-maximal representations} When the Milnor-Wood inequality is violated for the number $\textbf{eu}(\rho)$, we call the representation $\rho$ \textit{super-maximal}. By Theorem \ref{t:MilnorWood}, those are the representations $\rho: \pi_1 (\Sigma_{0,n}) \rightarrow G$ whose Euler class satisfies 
$$\euler(\rho) = n - 1 \text{ or } n . $$
Such representations only occur when 
\begin{equation} \label{eq: restriction} 2\pi (n-1) \leq  \Theta(\rho)  \leq 2\pi n . \end{equation}
Moreover, $\euler(\rho) = n$ iff $\rho$ is the trivial representation (sending everyone to the identity). However, the set of representations such that  $\euler(\rho) = n-1$ has non empty interior in the set $\text{Hom} (\pi_1(\Sigma_{g,n}), G)$. For instance, such representations can be constructed by considering a necklace of negative triangle groups with appropriate angles, see subsection \ref{ss: Hamiltonian action}. We first prove that super-maximal representations answer positively Bowditch's question \cite[Question C]{Bowditch} in a very strong way:

\begin{MonThm}\label{t: totally non hyperbolic}
Every super-maximal representation $\rho: \pi_1(\Sigma_{0,n}) \rightarrow G$ is totally non hyperbolic, namely every element of $\pi_1(\Sigma_{0,n})$ which is homotopic to a simple closed curve is mapped by $\rho$ to a non hyperbolic isometry of $G$. 
\end{MonThm}

Totally non hyperbolic representations were already known to exist when $g=0$ and $n=4$: Shinpei Baba observed that the representations lying in the compact component of the relative character varieties of $\text{PSL} (2,\mathbb R)$ discovered by Benedetto and Goldman \cite{BenedettoGoldman} have this property. 

Their existence in genus zero contrasts with the higher genus case: indeed, Gallo, Kapovich and Marden \cite[Part A]{GKM} showed (among other things) that a non elementary representation from the fundamental group of a closed surface with values in $\text{PSL}(2,\mathbb C)$ maps  a certain element of the fundamental group isotopic to a simple closed curve to a hyperbolic element. 

A consequence of Theorem \ref{t: totally non hyperbolic} together with  the work of Gueritaud-Kassel \cite{GK} is that a super-maximal representation is dominated by any Fuchsian representation (the holonomy of a complete metric of finite volume on $\Sigma_{0,n}$). To be more precise, let $l(g):= \inf_{x\in \H^+} d(x,g(x)) $ be the translation length of an element $g\in G$, and for every representation $\rho \in \text{Hom}(\pi_1(\Sigma_{g,n}), G)$, let $L_\rho: \pi_1(\Sigma_{0,n}) \to [0,\infty)$ be defined by $L_\rho (g) = l(\rho(g))$. Then for any couple $(j, \rho)$ formed by a Fuchsian representation $j$ and a super-maximal representation $\rho$, we have $L_\rho \leq L_j$. We deduce 

\begin{MonCoro} \label{c: compacity super maximal}
In the character variety $\Rep (\pi_1(\Sigma_{0,n}), G)$, the subset of super-maximal representations  is compact.
\end{MonCoro}

\subsection{Compact components in the relative character varieties}

Let us fix $\mathbf{\alpha} = (\alpha_1, \ldots, \alpha_n) \in (0,2\pi)^n$. We denote by $\Rep_{\bf{\alpha}}(\pi_1(\Sigma_{0,n}), G)$ the set of conjugacy classes of representations such that $\theta_i(\rho)= \alpha_i$. 
Because the $\alpha_i$'s are different from $0$ and $2\pi$ the space $\Rep_{\bf{\alpha}}(\Sigma_{0,n},G)$ has the structure of a smooth manifold and carries a natural symplectic form that has been constructed Goldman \cite{GoldmanSymplecticStructure}, building on works of Atiyah and Bott \cite{AtiyahBott}. Let $\Rep_{\bf{\alpha}}^{SM}(\Sigma_{0,n}, G)$ denote the set of super-maximal representations in $\Rep_{\bf{\alpha}}(\Sigma_{0,n}, G)$. Corollary \ref{c: compacity super maximal} implies that $\Rep_{\bf{\alpha}}^{SM}(\Sigma_{0,n}, G)$ forms a (possibly empty) compact connected component of $\Rep_{\bf{\alpha}}(\Sigma_{0,n}, G)$. We prove

\begin{MonThm} \label{t:SymplecticSuperMaximal}
If $2(n-1)\pi <\sum_{i=1}^n \alpha_i < 2n\pi$, then the space $\Rep_{\mathbf{\alpha}}^{SM}(\Sigma_{0,n}, G)$ is non-empty and symplectomorphic to $\ProjC{n-3}$, with a multiple of the Fubini--Study symplectic form whose total volume is
\[ \frac{(\pi\lambda)^{n-3}}{(n-3)!}~,\]
where
\[\lambda = \sum_{i=1}^n \alpha_i - 2(n-1)\pi~.\]
\end{MonThm}

The proof is done in subsection \ref{ss: Delzantpolytope}. It makes use of a faithful Hamiltonian action of the torus $(\mathbb R/\pi \mathbb Z)^{n-3}$ on $\Rep_{\bf{\alpha}}^{SM}(\Sigma_{0,n}, G)$, associated to a pair-of-pants decomposition of $\Sigma_{0,n}$. Delzant proved in \cite{Delzant88} that compact symplectic manifolds provided with a faithful Hamitonian action of a torus of half the dimension are classified by the image of their \emph{moment map}, which is a polytope satisfying certain arithmeticity conditions. Here we compute explicitly the Delzant polytope of our action and recognize the one corresponding to a natural action of $\left(\R/\pi \mathbb{Z}\right)^{n-3}$ on $\ProjC{n-3}$ with a certain multiple of the Fubini--Study symplectic form.

\subsection{Geometrization by $\mathbb H^-$-conifolds}
In Section \ref{s:Geometrization}, we show that super-maximal representations can be geometrized by $\mathbb H^-$-structures in a very strong way. In fact, the set of all possible geometrizations by a $\mathbb H^-$-structure is a copy of Teichm\"uller space $\mathcal T_{0,n}$.

\begin{MonThm}\label{t: geometrization}
Let $\rho: \pi_1(\Sigma_{0,n}) \rightarrow G$ be a super-maximal representation. Then either it is Abelian, or for every $\sigma \in \mathcal T_{0,n}$, there exists a unique $\rho$-equivariant map $\widetilde{\Sigma_{0,n}} \rightarrow \mathbb H^-$ which is holomorphic with respect to the complex structure $\sigma$. 
\end{MonThm}

In terms of non abelian Hodge theory, Theorem \ref{t: geometrization} can be rephrased by saying that a flat bundle over a punctured sphere with anti-supermaximal monodromy  is a variation of Hodge structure. 

This property characterizes the super-maximal representations. For instance, for maximal representations, namely those satisfying $\textbf{eu} (\rho )= \chi (\Sigma_{g,n})$, the isomonodromic space of conical $\mathbb H^+$-structures is discrete, as was proven by Mondello in \cite{Mondello}.  In a companion paper, we will adress the problem of the geometrization of representations that are merely maximal, using different techniques.  

Notice that Theorem \ref{t: geometrization}, together with the help of the Schwarz lemma, gives an alternative proof of the fact that super-maximal representations are dominated by Fuchsian representations. Also, the proof of Theorem \ref{t: geometrization} allows to find explicit parametrizations of super-maximal components by symmetric powers of the Riemann sphere (which are models for the complex projective spaces). These parametrizations transit via the Troyanov uniformization theorem.

\section{The refined Milnor--Wood inequality}\label{s: MilnorWood}

In this section, we establish Theorem \ref{t:MilnorWood}.

\subsection{Reduction to a bound from above}

Let us explain first that Theorem \ref{t:MilnorWood} is a consequence of the following result, whose proof will be postponed to the next two subsections.

\begin{prop}\label{p: reduction}  We have
\begin{equation}\label{eq: bound from above} \euler(\rho) \leq \sup \big( |\chi(\Sigma_{g,n}) | , \lfloor \frac{1}{2\pi}\Theta(\rho) \rfloor \big) \leq |\chi(\Sigma_{g,n})|+2~. \end{equation}
The inequality $\euler(\rho) > |\chi(\Sigma_{g,n})|$ is  possible only when $g=0$, none of the conjugacy classes are hyperbolic, and the volume is non-positive. Moreover, $\euler(\rho) = |\chi(\Sigma_{0,n})| + 2$ if and only if $\rho$ is the trivial representation.
\end{prop}

Let us prove Theorem \ref{t:MilnorWood} assuming Proposition \ref{p: reduction}. Let $\overline{\rho}$ be the conjugation of  $\rho$ by an orientation-reversing isometry of $\H^+$.  We have the formulas 
$$ \text{Vol} (\overline{\rho}) = - \text{Vol} (\rho)$$
and 
$$ \Theta (\rho ) + \Theta (\overline{\rho})  = 2 \pi l~, $$
where $l$ is the number of elliptic or parabolic cusps (counting the identity twice). In particular, we deduce 
\begin{equation}\label{eq: conjugation}   \euler(\overline{\rho}) = \frac{1}{2\pi} \big(  \text{Vol} (\overline{\rho}) + \Theta (\overline{\rho}) \big) = l + \mathbf{eu}(\rho) . \end{equation}
Applying \eqref{eq: bound from above} to $\overline{\rho}$, we get 
$$   \euler(\rho) \geq  \inf \big( -|\chi(\Sigma_{g,n}) |+ l  , \lceil  \frac{1}{2\pi}\Theta(\rho) \rceil  \big) ,$$ 
which concludes the proof of Theorem \ref{t:MilnorWood} if \eqref{eq: bound from above} holds.

\subsection{The case of the pair of pants}

In this paragraph we prove Proposition \ref{p: reduction} in the case of the pair of pants $\Sigma_{0,3}$.

The universal cover $\widetilde{G}$ of the group $G$ acts faithfully on $\widetilde{\mathbb R{\bf  P}^1}$. We denote by $m$ the generator of the covering group that acts positively with respect to the natural orientation of $\mathbb R {\bf P} ^1$, namely $m(x) >x$ for every $x\in \widetilde{\mathbb R {\bf P}^1}$. On $\widetilde{G}$ there is a well-defined notion of translation number: identifying $\widetilde{\mathbb R {\bf  P}^1}$ with $\mathbb R$ in such a way that $m$ is conjugated to the translation $x\mapsto x+ 2\pi$, the translation number $T(g)$ of an element $g\in \widetilde{G}$ is the following limit: 
$$ T(g) = \lim _{k\rightarrow \pm \infty} \frac{g^k (x)-x} {k} $$
It does not depend on $x\in \mathbb R$. We refer to \cite{Ghys} for a survey on this notion.

The fundamental group $\pi_1(\Sigma_{0,3})$ is generated by three peripheral elements $c_1, c_2, c_3$ that satisfy the relation $c_1 c_2 c_3= 1$, and that correspond to positively oriented loops around the punctures. For each $i= 1,2,3$, we denote by $\widetilde{\rho(c_i)}$ the unique lift of $\rho(c_i) $ in $\widetilde{G}$ having translation number $\theta_i(\rho)$. Notice that since $\rho (c_1)\rho(c_2)\rho(c_3)= 1$, there exists $k\in \mathbb Z$ such that 
$$ \widetilde{\rho(c_1)} \widetilde{\rho (c_2)} \widetilde{\rho (c_3)} = m ^k .$$

\begin{lem}\label{l:formula euler class} We have $\euler(\rho) = k$.\end{lem}

\begin{proof} Lift $\rho$ to a representation $\tilde{\rho}: \pi_1(\Sigma_{0,3}) \rightarrow \widetilde{G}$. Then for each $i$ there exists some integer $k_i$ such that 
$$ \widetilde{\rho(c_i)} = \tilde{\rho} (c_i) m^{k_i} .$$
We then have the relations 
$$ k_1 + k_2 + k_3 = k  \text{ and } \theta_i(\rho) = T (\tilde{\rho} (c_i) ) + k_i.$$
Both come from the fact that $m$ belongs to the center of $\widetilde{G}$. So 
$$ \sum_i \theta_i(\rho) = \sum_i T (\tilde{\rho} (c_i) ) + 2\pi k,$$
and the claim follows from \cite[Theorem 12]{BIW}.
\end{proof}

We now proceed to a case-by-case analysis. 


\subsubsection*{Case of an identity peripheral.} We first consider the case where $\rho(c_i) = \Id$ for some~$i$. Applying a cyclic permutation if necessary,  we can assume $\rho(c_3) = \Id$. Thus $\rho(c_2) = \rho(c_1)^{-1}$. We then have $\widetilde{\rho(c_3)} = m$ and 
\begin{displaymath}
\begin{array}{rcll}
m^k = \widetilde{\rho(c_1)} \widetilde{\rho(c_2)} \widetilde{\rho(c_3)} & = & m & \textrm{if $\rho(c_1),\rho(c_2)$ are hyperbolic}\\
 \ & = & m^2 & \textrm{if $\rho(c_1), \rho(c_2)$ are elliptic or parabolic}\\
 \ & = & m^3 & \textrm{if $\rho(c_1)=\rho(c_2)=1$.}
 \end{array}
\end{displaymath}
Notice that in the first case $\Theta(\rho) = 2\pi$, in the second case, $\Theta(\rho) =  4\pi$, and in the third case $\Theta(\rho) = 6\pi$. By Lemma \ref{l:formula euler class}, we thus have $\euler(\rho) = \frac{1}{\pi}\Theta(\rho)$, which proves the inequality \ref{eq: bound from above}.


\subsubsection*{Case of a hyperbolic peripheral.} Assume now that one of the $\rho(c_i)$'s is hyperbolic. Up to cyclic permutation, we can assume that $\rho(c_3)$ is hyperbolic. Notice that the lifts $\widetilde{\rho(c_i)}$ are chosen so that 
\begin{equation}\label{eq: properties of lift}   m^{-1} (y) < \widetilde{\rho (c_i)} (y) \leq m(y) \end{equation}
for every $y\in \widetilde{\mathbb R P^1}$. Moreover, there exist two points $x^\pm \in \widetilde{\mathbb R {\bf  P}^1}$ such that 
$$   \widetilde{ \rho(c_3) } (x ^+ ) > x^+ \text{ and }  \widetilde{ \rho(c_3) } (x ^- ) < x^-. $$
From \eqref{eq: properties of lift}, we deduce that 
$$ m^k (x^+) = \widetilde{ \rho(c_1) }\widetilde{ \rho(c_2) }\widetilde{ \rho(c_3) } (x^+)> m^{-2} (x^+)$$ 
and similarly 
$$ m^k (x^-) =  \widetilde{ \rho(c_1) }\widetilde{ \rho(c_2) }\widetilde{ \rho(c_3) } (x^-) < m^{2} (x^-).$$
The integer $k$ hence satisfies $|k|\leq 1$, which implies the proposition in the case one of the $\rho(c_i)$ is hyperbolic. \footnote{Note that there exists a representation having $\rho(c_1),\rho(c_2)$ negative parabolic and $\rho(c_3)$ hyperbolic:  such a representation has $\euler(\rho) = 1$ but $\Theta(\rho) = 4\pi$, showing that Theorem \ref{t:MilnorWood} is not sharp when some peripheral is parabolic.}


\subsubsection*{Case where none of the $\rho(c_i)$'s is identity or hyperbolic.}  In this case each $\rho (c_i)$ has a unique fixed point $p_i \in \mathbb H\cup \partial \mathbb H$. Either they are distinct, or equal.

We first consider the case where $p_1 = p_2 = p_3$. In this case, the volume is zero. Moreover, if $p_i$ lies in $\H$, then $\theta_i(\rho) <2\pi$ since none of the $\rho(c_i)$ are the identity. In particular $\Theta(\rho) = 2\pi$ or $4\pi$. If $p_i$ lies in $\partial \H$, then one of the $\rho(c_i)$ has to be a positive parabolic, since their product is $1$. In particular, $\Theta(\rho) = 2\pi$ or $4\pi$ as before. So we are done.

Suppose now that the $p_i$'s are distinct. In this case, the image of $\rho$ is a ``triangle group''. More precisely, for $p,q \in \mathbb H \cup \partial \mathbb H$ distincts, let $\sigma_{pq}$ be the reflection with respect to the geodesic $(pq)$. We then have the formulas 
$$ \rho(c_1) = \sigma_{p_3p_1} \sigma_{p_1p_2} ,\  \rho(c_2) = \sigma _{p_1p_2} \sigma_{p_2 p_3} \text{ and } \rho(c_3) = \sigma_{p_2p_3} \sigma _{p_3p_1}$$
(see Figure \ref{fig:TriangularRep}.)

\begin{figure}  
\includegraphics[width=9cm]{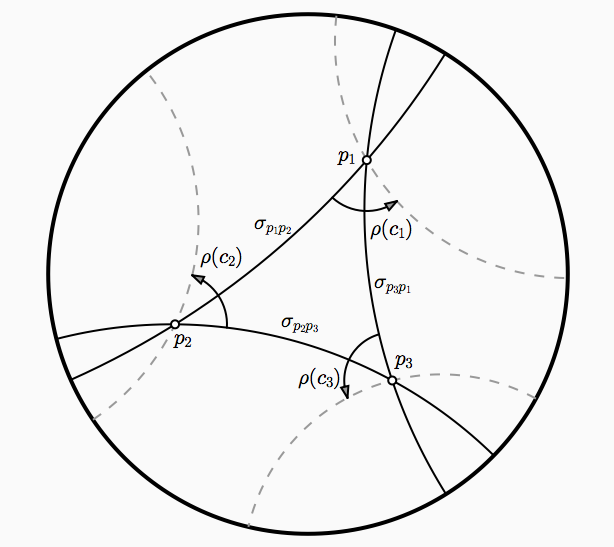}
\caption{A representation with the $\rho(c_i)$ elliptic. \label{fig:TriangularRep}}
\end{figure}

Since none of the $\rho(c_i)$'s is the identity, the triangle $ \Delta = p_1p_2p_3$ is non degenerate. In particular, $\rho$ is the holonomy of a positive (if $p_1p_2p_3$ is clockwise oriented) or negative (if $p_1p_2p_3$ is anti-clockwise oriented) hyperbolic metric on the sphere minus three points, obtained by gluing two copies of $\Delta$. In the first case, $\theta_i(\rho)$ is twice the angle of $\Delta$ at $p_i$, and the volume of $\rho$ twice the volume of $\Delta$, so by Gauss-Bonnet we get  $\euler(\rho) = 1$.  In the second case, $2\pi - \theta_i (\rho)$ is twice the angle of $\Delta$ at $p_i$, and the volume of $\rho$ is minus twice the volume of $\Delta$, so we get $ \euler (\rho) = 2$ in this case.  

The proof of Proposition \ref{p: reduction} in the case of $\Sigma_{0,3}$ is now complete.

\subsection{The general case}

Crucially, our induction will use the following fact:

\begin{prop}
Assume $\Sigma$ is obtained from a (possibly disconnected) surface $\Sigma'$ by guling $b$ with $b'^{-1}$, where $b$ and $b'$ are two boundary curves of $\Sigma'$. Let $\rho$ be a representation of $\pi_1(\Sigma)$ into $G$ and $\rho'$ the restriction of $\rho$ to $\pi_1(\Sigma')$. Then we have the following:
\begin{itemize}
\item if $\rho(b)$ is the identity, then
\[\euler(\rho) = \euler(\rho') -2~;\]
\item if $\rho(b)$ is parabolic or elliptic, then
\[\euler(\rho) = \euler(\rho') -1~;\]
\item if $\rho(b)$ is hyperbolic, then
\[\euler(\rho) = \euler(\rho')~.\]
\end{itemize}
(If $\Sigma'$ has several connected components, we denote my $\mathbf{eu}(\rho')$ the sum of the Euler classes of the restrictions of $\rho$ to the fundamental group of each connected component.)
\end{prop}

\begin{proof}
By additivity of the volume, we have $\Vol(\rho) = \Vol(\rho')$. Therefore,
\[\euler(\rho') = \euler(\rho) + \frac{1}{2\pi}\left(\theta(\rho'(b)) + \theta(\rho'(b'))\right)~.\]
One has $\rho'(b') = \rho'(b)^{-1}$ and therefore
\begin{displaymath}
\begin{array}{rcll}
\theta(\rho'(b)) + \theta(\rho'(b')) & = & 4\pi & \textrm{if $\rho'(b)$ is the identity}\\
 \ & = & 2\pi & \textrm{if $\rho'(b)$ is elliptic or parabolic}\\
 \ & = & 0 & \textrm{if $\rho'(b)$ is hyperbolic.}
 \end{array}
\end{displaymath}
\end{proof}

Let us now prove Proposition \ref{p: reduction} for $\Sigma_{0,n}$ by induction on $n$. 


\begin{proof}[Proof of Proposition \ref{p: reduction} in the genus $g=0$ case]  

We decompose $\Sigma = \Sigma_{0,n}$ as the union of $\Sigma'= \Sigma_{0,k+1}$ and $\Sigma''= \Sigma_{0,n-k+1}$, where some boundary curve $b'$ of $\Sigma_{0,k+1}$ is glued with $b''^{-1}$ for some boundary curve $b''$ of $\Sigma_{0,n-k+1}$. Denote by $\rho'$ the restriction of $\rho$ to $\pi_1(\Sigma_{0,k+1})$ and by $\rho''$ the restriction of $\rho$ to $\pi_1(\Sigma_{0,n-k+1})$.

By induction, we have $\euler(\rho') \leq k+1$ and $\euler (\rho'') \leq n-k+1$. We can now proceed to a case-by-case study:
\begin{itemize}
\item If $\euler (\rho') = |\chi(\Sigma') |+2$ and $\euler( \rho'' ) = |\chi(\Sigma'') |+2$, then by induction hypothesis they are both trivial. Hence $\rho$ is trivial and $\euler(\rho) = n$.
\item If $\euler(\rho') = |\chi(\Sigma') |+2$ and $\euler(\rho'') = |\chi(\Sigma'') |+1$ (or the converse) then, by induction hypothesis, $\rho'$ is trivial. Therefore $\rho'(b)$ is the identity. Hence
\[\euler(\rho) = \euler(\rho') \euler(\rho'') -2 = n-1~.\]
By induction, no boundary curve of $\Sigma_{0,k+1}$ and $\Sigma_{0,n-k+1}$ has hyperbolic image. Hence the same holds for $\Sigma_{0,n}$.
 Moreover, we have $\lfloor \sum_{i'} \theta_{i'}(\rho') \rfloor = |\chi(\Sigma')|+ 2$ and $\lfloor \sum_{i''} \theta_{i''}(\rho'') \rfloor = |\chi(\Sigma'')|+1$, so 
$$\lfloor \sum_{i} \theta_{i}(\rho) \rfloor =\lfloor \sum_{i'} \theta_{i'}(\rho') + \sum_{i''} \theta_{i''}(\rho'')  - 2 \rfloor = n-1~,$$ 
proving \eqref{eq: bound from above}. 
\item If $\euler(\rho') = |\chi(\Sigma') |+1$ and $\euler(\rho'') = |\chi(\Sigma'') |+1$, then $\rho'(b)$ is not hyperbolic. If $\rho'(b)$ is the identity, then
\[\euler(\rho) = \euler(\rho') \euler(\rho'') -2 = n-2~.\]
Otherwise, we have
\[\euler(\rho) \leq \euler(\rho') \euler(\rho'') -1 = n-1~.\]
Like in the previous case, no boundary curve of $\Sigma$ is sent to a hyperbolic element. Finally, since $\rho(b'), \rho(b'')$ are not hyperbolic, we get $\theta(\rho(b')) + \theta(\rho(b'')) = 2\pi$, and   
\begin{eqnarray*}
\sum_{i} \theta_{i}(\rho) & = & \sum_{i'} \theta_{i'}(\rho') + \sum_{i''} \theta_{i''}(\rho'')  - 1 \\
& \geq & (|\chi(\Sigma')| + 1) + (|\chi (\Sigma'')|+ 1) - 1 = |\chi (\Sigma)| + 1~,
\end{eqnarray*}
proving \eqref{eq: bound from above}.  
\item If $\euler(\rho') = |\chi(\Sigma') |+2$ and $\euler(\rho'') \leq |\chi(\Sigma'')|$ (or the converse), then $\rho(b')$ is the identity. Hence
\[\euler(\rho) = \euler(\rho') \euler(\rho'') -2 \leq n-2~.\]
\item If $\euler(\rho') = |\chi(\Sigma') |+1$ and $\euler(\rho'') \leq |\chi(\Sigma'')|$ (or the converse), then $\rho(b')$ is not hyperbolic. Hence
\[\euler(\rho) \leq \euler(\rho') \euler(\rho'') -1 \leq n-2~.\]
\item Finally, if $\euler(\rho') \leq |\chi(\Sigma') |$ and $\euler(\rho'') \leq |\chi(\Sigma'')|$, then
\[\euler(\rho) \leq \euler(\rho') \euler(\rho'') \leq n-2~.\]
\end{itemize}
\end{proof}

We can now prove Proposition \ref{p: reduction} in the higher genus case. Note that
\[ \frac{1}{2\pi} \Theta(\rho) = \frac{1}{2\pi}\sum_i \theta_i(\rho) \leq n \leq 2g-2+n = |\chi(\Sigma_{g,n})|\]
as soon as $g\geq 1$. Therefore, Proposition \ref{p: reduction} when $g \geq 1$ reduces to the classical Milnor-Wood inequality
\[\euler(\rho) \leq |\chi (\Sigma_{g,n}) |~.\]


\begin{proof}[Proof of Proposition \ref{p: reduction} if $g>0$]
We argue by induction on $g$.
\begin{itemize}
\item[$g=1$.]
The surface $\Sigma_{1,n}$ is obtained from $\Sigma_{0,n+2}$ by gluing together $b$ and ${b'}^{-1}$, for two boundary curves $b$ and $b'$. Denote by $\rho'$ the restriction of $\rho$ to $\pi_1(\Sigma_{0,n+2})$. If $\euler(\rho') = n+2$, then $\rho'$ is trivial, hence $\rho$ is trivial and $\euler(\rho) = n$. If $\euler(\rho') = n+1$, then $\rho'(b)$ is not hyperbolic. Hence
\[\euler(\rho) \leq \euler(\rho')-1 = n~.\]
Finally, if $\euler(\rho') \leq n$ then $\euler(\rho) \leq n$.\\

\item[$g\geq 2$.]
The surface  $\Sigma_{g,n}$ is obtained from $\Sigma_{g-1,n+2}$ by gluing together two boundary curves $b$ and $b'$. Denote by $\rho'$ the restriction of $\rho$ to $\pi_1(\Sigma_{g-1,n+2})$. By induction hypothesis, we have
\[\euler(\rho') \leq 2(g-1)-2 + n+2 = 2g-2+n~,\]
hence
\[\euler(\rho) \leq \euler(\rho') \leq 2g-2+n~.\]
\end{itemize}
\end{proof}

\section{Super-maximal representations}

\subsection{Definition and examples} Super-maximal representations are representations whose Euler class violates the classical Milnor-Wood inequality. It happens only in the following situation

\begin{defi}
A representation $\rho: \pi_1(\Sigma_{g,n}) \to G$ is called \emph{super-maximal} if $g=0$ and $\euler (\rho) = n-1$ ou $\euler (\rho)= n$.
\end{defi}

As we saw in Proposition \ref{p: reduction}, super-maximal representations have Euler class $n-1$, except for the trivial representation which has Euler class $n$.


\subsection{Super-maximal representations are ``totally non hyperbolic''}

A first important fact about super-maximal representations is that they send every simple closed curve to a non hyperbolic element.

\begin{prop} \label{p:SuperMaxElliptic}
Let $\rho: \pi_1(\Sigma_{0,n}) \to G$ be a representation of Euler class $\euler(\rho) = n-1$ or $n$. Then, for any element $\gamma$ in $\pi_1(\Sigma_{0,n})$ freely homotopic to  a simple closed curve, $\rho(\gamma)$ is not hyperbolic.
\end{prop}

\begin{proof}
Let $\gamma$ be a simple closed curve in $\Sigma_{0,n}$. If $\gamma$ is freely homotopic to a boundary curve, then $\rho(\gamma)$ is not hyperbolic, as part of Proposition \ref{p: reduction}. Othewise, $\gamma$ cuts $\Sigma_{0,n}$ into two surfaces $\Sigma'$ and $\Sigma''$. We saw in the demonstration of Proposition \ref{p: reduction} that the restrictions of $\rho$ to $\pi_1(\Sigma')$ and $\pi_1(\Sigma'')$ are both super-maximal, and therefore that the images of the boundary curves of $\Sigma'$ by $\rho$ (and in particular $\rho(\gamma)$) are non hyperbolic.
\end{proof}

\begin{rmk}
Shinpei Baba observed that the representations lying in the Benedetto--Goldman compact components \cite{BenedettoGoldman} are totally elliptic. This remark gave us the idea that super-maximal representations should form compact components of relative character varieties.
\end{rmk}

\subsection{Domination} As a corollary, we obtain that $\rho$ is dominated by any Fuchsian representation:

\begin{coro} \label{c:DominationSuperMax}
Let $\rho: \pi_1(\Sigma_{0,n}) \to G$ be a representation of Euler class $n-1$ or $n$. Then, for any Fuchsian representation $j: \pi_1(\Sigma_{0,n}) \to G$, there exists a $1$-Lipschitz $(j,\rho)$-equivariant map from $\H$ to $\H$.
\end{coro}

\begin{proof}
For $g\in G$, let $L(g)$ denote the translation length of $g$ (seen as an isometry of $\H^2$). According to the work of Gu\'eritaud--Kassel \cite{GK}, in order to obtain the conclusion, it is enough to know that
\[L_\rho(\gamma)) \leq L_j(\gamma))\]
for every simple closed curve $\gamma$. This is obviously true since, by Proposition~\ref{p:SuperMaxElliptic}, $L_\rho(\gamma)) = 0$ for every simple closed curve $\gamma$.
\end{proof}

As a consequence, we obtain that the length spectrum $L_\rho$ of any super-maximal representation is bounded by the following interesting function 
\[C (\gamma ) = \inf\{L_j(\gamma)), j:\pi_1(\Sigma_{0,n}) \to G \textrm{ Fuchsian}\}~.\]
It is invariant by conjugation and by the braid group, and measures a certain complexity of the corresponding element of $\pi_1(\Sigma_{0,n})$. It would be interesting to understand this function in more details. See Basmajian \cite{Basmajian13} for related results.

\subsection{Compacity of the space of super-maximal representations}

Corollary \ref{c:DominationSuperMax} provides a uniform control on all super-maximal representations, from which we can prove that the space of super-maximal representations is compact.

Let $\Hom(\Sigma_{0,n}, G)$ denote the space of representations of $\pi_1(\Sigma_{0,n})$ into $G$ and $\Rep(\Sigma_{0,n}, G)$ its quotient by the action of $G$ by conjugation. The natural topology of $\Hom(\Sigma_{0,n}, G)$ induces a non-Hausdorff topology on $\Rep(\Sigma_{0,n}, G)$. 

\begin{prop} \label{p:CompacitySuperMax}
The space of super-maximal representations is a compact subset of $\Rep(\Sigma_{0,n}, G)$.
\end{prop}

\begin{proof}
Let $(\rho_n)_{n\in \N}$ be a sequence of super-maximal representations. We fix a Fuchsian representation $j$. By Corollary \ref{c:DominationSuperMax}, we can find a sequence of $1$-Lipschitz maps $f_n: \H \to \H$ such that $f_n$ is $(j,\rho_n)$-equivariant. Up to conjugating each $\rho_n$ and composing each $f_n$ by an isometry of $\H$, we can assume that each $f_n$ fixes a given base point. By Ascoli's theorem, up to extracting a subsequence, $f_n$ converges uniformly on every compact set to a $1$-Lipschitz map $f_\infty$. This map $f_\infty$ is $(j,\rho_\infty)$-equivariant for some representation $\rho_\infty$ and we have
\[\rho_n \underset{n\to +\infty}{\to} \rho_\infty\]
in $\Rep(\Sigma_{0,n}, G)$. 

Finally, since the function $\euler: \Rep(\Sigma_{0,n}, G) \to \R$ is upper semi-continuous, the limit $\rho_\infty$ is still super-maximal. We have thus proved that the set
\[\{ \rho \in \Rep(\Sigma_{0,n}, G) \to \R \mid \euler(\rho) \geq n-1\}\]
is sequentially compact, hence compact. \end{proof}

\subsection{Compact components in relative character varieties} 

Recall that $\Rep_{\bf{\alpha}}^{SM}(\Sigma_{0,n}, G)$ denotes the set of super-maximal representations in the corresponding relative character variety $\Rep_{\bf{\alpha}}(\Sigma_{0,n}, G)$. 

\begin{prop}
If $2(n-1)\pi <\sum_{i=1}^n \alpha_i < 2n\pi$, then $\Rep_{\bf{\alpha}}^{SM}(\Sigma_{0,n}, G)$ forms a non-empty compact connected component of $\Rep_{\bf{\alpha}}(\Sigma_{0,n}, G)$.
\end{prop}

\begin{proof}
Since $\Rep_{\bf{\alpha}}(\Sigma_{0,n}, G)$ is closed in $\Rep(\Sigma_{0,n}, G)$, its intersection with the set of super-maximal representations is compact by Proposition \ref{p:CompacitySuperMax}. Since none of the $\alpha_i$ is equal to $0$ or $2\pi$, the Euler class is continuous in restriction to $\Rep_{\bf{\alpha}}(\Sigma_{0,n}, G)$. Since it takes integral values, the subset of super-maximal representations is a union of connected components. 

It remains to prove that $\Rep_{\bf{\alpha}}^{SM}(\Sigma_{0,n}, G)$ is non-empty and connected. We postpone this to section \ref{ss: Delzantpolytope}.
\end{proof}

\section{Symplectic geometry of super-maximal components}

\subsection{The Goldman symplectic structure on (relative) character varieties}

Goldman constructed in \cite{GoldmanSymplecticStructure} a natural symplectic structure on the character variety $\chi_G(\Sigma)$ of the fundamental group of a closed connected oriented surface $\Sigma$ into $G$, and in fact into any semi-simple Lie group. Moreover, he found in \cite{Goldmanduality} a duality between conjugacy invariant functions on $G$ and certain ``twisting'' deformations of representations.

More precisely, let $F: G \to \R$ be a fonction invariant by conjugation. Recall that there is a natural non-degenerate bilinear form $\Kill_G$ on the Lie algebra $\g$ which is invariant under the adjoint action of $G$: the Killing form. At a point $g\in G$ where $F$ is $\mathcal{C}^1$, we define $\delta_g F$ as the vector in $\g$ such that
\[ \d F_g(g\cdot v) = \Kill_G(\delta_g F, v)\]
for all $v\in \g$.
Because $F$ is invariant by conjugation, $\delta_g F$ is centralized by $g$.

Now, let $\rho: \pi_1(\Sigma) \to G$ be a representation. Let $b$ denote a simple closed curve in $\Sigma$ which is not homotopic to a boundary curve. If $b$ is separating, then it cuts $\Sigma$ into two surfaces $\Sigma'$ and $\Sigma''$ and we can write
\[\pi_1(\Sigma) = \pi_1(\Sigma') * \pi_1(\Sigma'')/b' \sim b''~.\]
If $b$ is not separating, cutting along $b$ gives a compact surface $\Sigma'$ and we can write
\[\pi_1(\Sigma) = \pi_1(\Sigma') * \langle u \rangle/b_{left}\sim u b_{right}u^{-1}~.\]

Since $\rho(b)$ centralizes $\delta_{\rho(b)} F$, we can ``twist'' the representation $\rho$ along $b$ and define a representation $\Phi_{F,b,t}(\rho)$ by
\[\begin{array}{lccc}
\Phi_{F,b,t}(\rho): & \gamma \in \pi_1(\Sigma') & \mapsto & \rho(\gamma) \\
                  & \gamma \in \pi_1(\Sigma'')& \mapsto & \exp(t \delta_{\rho(b)} F) \rho(\gamma) \exp(-t \delta_{\rho(b)} F)
\end{array}\]
if $b$ is separating and
\[\begin{array}{lccc}
\Phi_{F,b,t}(\rho): & \gamma \in \pi_1(\Sigma') & \mapsto & \rho(\gamma) \\
                  & u                         & \mapsto & u \exp(t \delta_{\rho(b)} F)
\end{array}\]
if $b$ is non-separating. On can prove that $\Phi_{F,b,t}$ induces a well-defined flow on the character variety of $\Sigma$.

\begin{CiteThm}[Goldman] \label{t:Goldmanduality}
Let $F$ be a function of class $\mathcal{C}^1$ on $G$ invariant by conjugation. If $b$ is a simple closed curve in $\Sigma$. Denote by $\mathbf{F}_b$ the function on $\chi_G(\Sigma)$ defined by
\[\mathbf{F}_b(\rho) = F(\rho(b))~.\]
Then $\mathbf{F}_b$ is $\mathcal{C}^1$ and its Hamiltonian flow (with respect to the Goldman symplectic form) is the flow $\left(\Phi_{F,b,t}\right)_{t\in \mathbb{R}}$.
\end{CiteThm}

This generalizes to relative character varieties of surfaces with boundary (see for instance \cite{Goldman06}).

\subsection{A Hamiltonian action of $\left(\R/\pi \mathbb{Z}\right)^{n-3}$} \label{ss: Hamiltonian action}

We consider the decomposition of $\Sigma_{0,n}$ into pairs of pants given by Figure \ref{fig:POPDecomposition}.

\begin{center}
\begin{figure} 
\includegraphics[width=12cm]{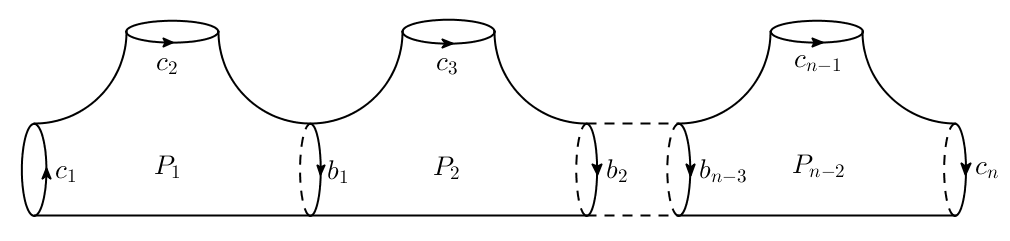}
\caption{A pair-of-pants decomposition of $\Sigma_{0,n}$. \label{fig:POPDecomposition}}
\end{figure}
\end{center}

If $\rho$ is a representation of $\pi_1(\Sigma_{0,n})$ into $G$, we note $\beta_i(\rho)= \theta(\rho(b_i))$. We also set $\bar{\alpha}_i = 2\pi - \alpha_i$.

\begin{lem}
If $\rho \in \Rep_{\bf{\alpha}}^{SM}(\Sigma_{0,n}, G)$, then, for all $1\leq i \leq n-3$, we have
\[ \sum_{k=1}^{i+1} \bar{\alpha}_k \leq \beta_i(\rho) \leq 2\pi - \sum_{k=i+2}^n \bar{\alpha}_k~.\]
In particular, $\rho(b_i)$ is elliptic.
\end{lem}

\begin{proof}
The curve $b_i$ cuts $\Sigma_{0,n}$ into two surfaces $\Sigma'_{0,i+2}$ and $\Sigma''_{0,n-i}$. Let $\rho'$ and $\rho''$ denote respectively the restrictions of $\rho$ to $\pi_1(\Sigma')$ and $\pi_1(\Sigma'')$. Since none of the $\alpha_i$ is equal to $2\pi$, neither $\rho'$ nor $\rho''$ are trivial. Since $\rho$ is super-maximal, $\rho'$ and $\rho''$ are also super-maximal. Applying Proposition \ref{p: reduction} to $\rho'$, we get
\[\beta_i(\rho) + \sum_{k=1}^{i+1} \alpha_k \geq 2\pi(i+1)~.\]
Applying Proposition \ref{p: reduction} to $\rho''$, we get
\[\bar{\beta}_i(\rho) + \sum_{k=i+2}^n \alpha_k \geq 2\pi(n-i-1)~,\]
where $\bar{\beta}_i(\rho) = \theta(\rho(b_i)^{-1})$.

If $\rho(b_i)$ were the identity, then we would get
\[\euler(\rho) = \euler(\rho') \euler(\rho'')-2 = |\chi(\Sigma_{0,n})|\]
and $\rho$ would not be super-maximal. 

Therefore, $\bar{\beta}_i(\rho) = 2\pi - \beta_i(\rho)$ and we get
\[ \sum_{k=1}^{i+1} \bar{\alpha}_k \leq \beta_i(\rho) \leq 2\pi - \sum_{k=i+2}^n \bar{\alpha}_k~.\]
In particular, $\rho(b_i)$ is elliptic.
\end{proof}

Since $\rho(b_i)$ is elliptic for all $\rho \in \Rep_\alpha ^{SM}  (\Sigma_{0,n}, G)$, The functions $\beta_i$ are $n-3$ well-defined smooth functions on $\Rep_\alpha ^{SM} (\Sigma_{0,n}, G)$.

\begin{prop}
The Hamiltonian flow associated to the function $\beta_i$ is $\pi$-periodic.
\end{prop}

\begin{proof}
Recall that the function $\theta$ on $G$ associates to an elliptic element $g$ the rotation angle of $g$ in $\H^+$. The function $\theta$ is invariant by conjugation. Let us compute its gradient (with respect to the Killing metric) at the point
\[g_0 = \left(\begin{matrix} \cos(\theta_0/2) & - \sin(\theta_0/2) \\ \sin(\theta_0/2) & \cos(\theta_0/2)\end{matrix} \right)~.\]
We can write 
\[\theta(g) = f(\Tr(g))~,\]
where $f$ satisfies
\[f(2\cos(x/2)) = x~.\]

For $u\in \sl(2,\R)$, we have
\begin{eqnarray*}
\dt_{|t=0} f\left(\Tr(g_0 \exp(tu))\right) & = & f'\left(\Tr(g_0) \right) \Tr(g_0 u)\\
\ & = & f'(2 \cos(\theta_0/2)) \Tr\left( (g_0 - \Tr(g_0) \I_2)u\right)~.
\end{eqnarray*}
Since $g_0 - \Tr(g_0) \I_2 \in \sl(2,\R)$, we deduce that
\begin{eqnarray*}
\delta_{g_0} \theta & = & f'(2 \cos(\theta_0/2)) (g_0 - \Tr(g_0)\I_2)\\
& = & \left(\begin{matrix} 0 & -f'(2\cos(\theta_0/2)) \sin(\theta_0/2)\\f'(2\cos(\theta_0/2)) \sin(\theta_0/2) & 0\end{matrix}\right)\\
& = &  \left(\begin{matrix} 0 & 1 \\-1 & 0\end{matrix}\right)
\end{eqnarray*}
since $-f'(2\cos(\theta_0/2)) \sin(\theta_0/2) = \frac{\d}{\d x}_{|x = \beta_0} f(2\cos(x/2)) = 1$.

Therefore, the flow
\[t\mapsto \exp(t \delta_{g_0} \theta)\]
is $\pi$-periodic. 

Since every elliptic element is conjugate to some $g_0$, we obtain thanks to Theorem \ref{t:Goldmanduality} that the Hamiltonian flow
\[\Phi_i = \Phi_{\theta,b_i}\]
associated to the function $\beta_i: \rho \mapsto \theta(\rho(b_i))$ on $\Rep_{\bf{\alpha}}^{SM}(\Sigma_{0,n}, G)$ is $\pi$-periodic.
\end{proof}

Since the curves $b_i$ are pairwise disjoint, the functions $\beta_i$ Poisson commute and their Hamiltonian flows together provide an action of $(\R/\pi \Z)^{n-3}$ on $\Rep_\alpha (\Sigma_{0,n}, G)$.

\subsection{The Delzant polytope of the Hamiltonian action} \label{ss: Delzantpolytope}

In this subsection we prove Theorem \ref{t:SymplecticSuperMaximal}. 

By the work of Delzant \cite{Delzant88}, in order to understand the symplectic structure of the manifold $\Rep_{\bf{\alpha}}^{SM}(\Sigma_{0,n}, G)$, it is essentially enough to understand the image of the \emph{moment map}:
\[\function{\mathbf{\beta}}{\Rep_\alpha^{SM} (\Sigma_{0,n}, G)}{\R^{n-3}}{\rho}{\left(\beta_1(\rho), \ldots , \beta_{n-3}(\rho)\right)}~.\]
More precisely, Delzant proved the following:
\begin{CiteThm}[Delzant, \cite{Delzant88}]
Let $(M,\omega)$ and $(M',\omega')$ be two compact symplectic manifolds of dimension $2(n-3)$ provided with a Hamiltonian action of $(\R/\pi\Z)^{n-3}$. Assume that the moment maps $\mathbf{\beta}$ and $\mathbf{\beta}'$ of these actions have the same image (up to translation). Then there exists a symplectomorphism $\phi:(M,\omega) \to (M',\omega')$ that conjugates the actions of $(\R/\pi\Z)^{n-3}$.
\end{CiteThm}

\begin{lem} \label{l:InequalitiesMomentMap}
The moment map $\mathbf{\beta}$ of the Hamiltonian action of $\R/\pi\Z$ described in Subsection \ref{ss: Hamiltonian action} satisfies the following $n-2$ affine inequalities:
\begin{eqnarray}
\label{eq:IneqMP1} \beta_1  & \geq & \bar{\alpha}_1 + \bar{\alpha}_2~,\\
\label{eq:IneqMP2} \beta_i - \beta_{i-1} & \geq & \bar{\alpha}_{i+1},\quad 2\leq i \leq n-3~,\\
\label{eq:IneqMP3} \beta_{n-3} & \leq & 2\pi - \bar{\alpha}_{n-1} - \bar{\alpha}_n~.
\end{eqnarray}

Conversely, if $(x_1, \ldots, x_{n-3}) \in \R_+^{n-3}$ satisfies the inequalities \eqref{eq:IneqMP1}, \eqref{eq:IneqMP2}, \eqref{eq:IneqMP3} (when substituting $x_i$ to $\beta_i$), then the set of super-maximal representations $\rho$ satisfying $\beta_i(\rho) = x_i$ is non-empty and connected.
\end{lem}

\begin{proof}
Let $P_1, \ldots, P_{n-2}$ denote the pants in the pair-of-pants decomposition given in Figure \ref{fig:POPDecomposition}. Let $\rho$ be a representation in $\Rep_{\bf{\alpha}}^{SM}(\Sigma_{0,n}, G)$. Then, as in the proof of Proposition \ref{p: reduction}, the restriction of $\rho$ to $\pi_1(P_i)$ is super-maximal for every $i$. In particular, the sum of the rotation angles of the images of the boundary curves of $P_i$ (with the proper choice of orientation) is at least $4\pi$. Applying this to each $P_i$ gives the required inequalities.

Conversely, let $x_1, \ldots, x_{n-3}$ satisfy the inequalities of Lemma \ref{l:InequalitiesMomentMap}. Fix $2\leq i \leq n-3$. Then we have $4 \pi \leq x_i + (2\pi - x_{i-1}) + \alpha_{i+1} < 6 \pi$. Therefore, there exists a super-maximal representation $\rho_i$ of $\pi_1(P_i)$ sending the boundary curves $\bar{b}_{i-1}$, $b_i$ and $c_{i+1}$ respectively to rotations of angle $2\pi - x_{i-1}$, $x_i$ and $\alpha_{i+1}$. Moreover, this representation is unique up to conjugation. it is obtained by considering a hyperbolic triangle $p_1p_2p_3$ oriented clockwise with angles $\pi - x_{i-1}/2$, $x_i/2$ and $\alpha_{i+1}/2$, and setting $\rho_i(\bar{b}_{i-1}) = \sigma_{p_3p_1}\sigma_{p_1p_2}$, $\rho_i(b_i) = \sigma_{p_1p_2} \sigma_{p_2p_3}$ and $\rho_i(c_{i+1}) = \sigma_{p_2 p_3} \sigma_{p_3 p_1}$ (cf Figure \ref{fig:TriangularRep}).

Similarly, there is a  representation $\rho_1$ (resp. $\rho_{n-2}$) of $\pi_1(P_1)$ (resp. $\pi_1(P_{n-2})$) satifying 
\[\left(\theta(\rho_1(c_1)), \theta(\rho_1(c_2)), \theta(\rho_1(b_1)) \right) = (\alpha_1, \alpha_2, x_1)\] 
(resp. \[\left(\theta(\rho_{n-2}(c_1)), \theta(\rho_{n-2}(c_2)), \theta(\rho_{n-2}(b_1)) \right) = (2\pi - x_{n-3}, \alpha_{n-1}, \alpha_n)~.\]

Now, there is a way to conjugate the $\rho_i$ so that they can be glued together to form a super-maximal representation $\rho$ satisfying $\theta(\rho(c_i)) = \alpha_i$ and $\theta(\rho(b_i)) = x_i$. More precisely, on can choose $g_1$ in $G$, and then recursively choose $g_{i+1} \in G$ such that
\[g_{i+1}\rho_{i+1}(\bar{b}_i)^{-1}g_{i+1}^{-1} = g_i \rho_i(b_i) g_i^{-1}~.\]
(This is possible because both $\rho_{i+1}(\bar{b}_i)^{-1}$ and $\rho_i(b_i)$ are rotations of angle $x_i$.)
There exists a representation $\rho$ whose restriction to each $\pi_1(P_i)$ gives $\Ad_{g_i} \circ \rho_i(b_i)$. Since the restriction of $\rho$ to each $\pi_1(P_i)$ is super-maximal, and since $\rho(b_i)$ is never trivial, the representation $\rho$ itself is super-maximal.

Finally, two choices of $g_{i+1}$ coincide up to left multiplication by an element of the centralizer of $\rho_{i+1}(\bar{b}_i)$ which is connected. Therefore the space of all choices of the $(g_1,\ldots, g_{n-2})$ is connected. Since any super-maximal representation $\rho$ in $\mathbf{\beta}^{-1}(x_1, \ldots, x_{n-3})$ is obtained by such a gluing, it follows that the fiber $\mathbf{\beta}^{-1}(x_1, \ldots, x_{n-3})$ is connected.
\end{proof}

It remains to identify the polytope defined by the equalities \ref{eq:IneqMP1}, \ref{eq:IneqMP2}, \ref{eq:IneqMP3} to the Delzant polytope of a certain torus action on $\ProjC{n-3}$. 

Recall that $\ProjC{n-3}$ carries a natural K\"ahler form $\omega_{FS}$.
There are $n-3$ natural commuting Hamiltonian actions $r_1,\ldots , r_{n-3}$ of $\R/\pi\Z$ on $\ProjC{n-3}$, given by
\[r_k(\theta) \cdot [z_0, \ldots, z_{n-3}] = [z_0, \ldots, z_{k-1}, e^{2i \theta} z_k, z_{k+1}, \ldots, z_n]~.\]
With a convenient scaling of $\omega_{FS}$, a moment map of the action $r_k$ with respect to the Fubini--Study symplectic form is the function
\[\function{\mu_k}{\ProjC{n-3}}{\R}{[z_0, \ldots, z_{n-3}]}{ \frac{|z_k|^2}{\sum_{j=0}^{n-3} |z_j|^2}}~.\]
(See \cite[Example p.317]{Delzant88}.)

The image of the moment map $\mathbf{\mu} = (\mu_1, \ldots, \mu_{n-3})$ is the symplex
\[\{(x_1, \ldots , x_{n-3}) \in \R_+^{n-3}\mid x_1 + \ldots + x_{n-3} \leq 1 ~\}.\]
Though this is not exactly the same symplex as the image of the moment map $\mathbf{\beta}$, it is identical up to translation, dilation, and a linear transformation in $\SL(n,\Z)$.

To be more precise, let us set $\lambda = 2\pi - \sum_{j=1}^n \bar{\alpha}_i$. Let us define an action $r'_k$ of $\R/\pi\Z$ on $\ProjC{n-3}$ by
\[r'_k(\theta) \cdot [z_0, \ldots, z_{n-3}] = [e^{2i\theta} z_0 , \ldots , e^{2i\theta} z_k, z_{k+1}, \ldots, z_n]~.\]
Then the function
\[\mu'_k = \lambda \sum_{j=1}^k \mu_j + \sum_{j=1}^{k+1} \alpha_j\]
is a moment map for the action $r'_k$ with respect to the symplectic form $\lambda \omega_{FS}$. The actions $r'_1, \ldots, r'_{n-3}$ still commute and provide a new action of $(\R/\pi\Z)^{n-3}$ with moment map
\[\mathbf{\mu}' = (\mu'_1, \ldots, \mu'_{n-3})~.\]
Given the affine relation between $\mathbf{\mu}$ and $\mathbf{\mu'}$, one sees that the image of $\mathbf{\mu'}$ is the set of vectors $(x_1, \ldots, x_{n-3})$ in $\R^{n-3}$ satisfying
\begin{eqnarray*}
x_1 & \geq & \bar{\alpha}_1 + \bar{\alpha}_2~,\\
x_j & \geq & x_{j-1} + \bar{\alpha}_{j+1},\quad 2\leq j \leq n-3~,\\
x_{n-3} & \leq & \lambda + \sum_{j=1}^{n-2} \bar{\alpha_j} = 2\pi - \bar{\alpha}_{n-1} - \bar{\alpha}_n~.
\end{eqnarray*}
These are exactly the inequalities \ref{eq:IneqMP1}, \ref{eq:IneqMP2}, \ref{eq:IneqMP3}. By Delzant's theorem, it follows that $\Rep_{\mathbf{\alpha}}^{SM}(\Sigma_{0,n}, G)$ is isomorphic (as a symplectic manifold with a Hamiltonian action of $(\R/\pi\Z)^{n-3}$) to $(\ProjC{n-3}, \lambda \omega_{FS})$ with the action $(r'_1, \ldots, r'_{n-3})$.

In particular, the symplectic volume of $\Rep_{\mathbf{\alpha}}^{SM}(\Sigma_{0,n}, G)$ is equal to 
\[\left(\lambda \right)^{n-3} \int_{\ProjC{n-3}} \omega_{FS}^{n-3}~.\] 
Using the Hamiltonian action $(r_1, \ldots, r_{n-3})$ of $\left(\R/\pi \Z\right)^{n-3}$, we see that
\begin{eqnarray*}
\int_{\ProjC{n-3}} \omega_{FS}^{n-3} &=& \pi^{n-3} \int_{x_i \geq 0, \sum x_i \leq 1} \d x_1 \ldots \d x_{n-3}\\
\ & = & \frac{\pi^{n-3}}{(n-3)!}~.
\end{eqnarray*}

Thus 
\[\int_{\Rep_{\mathbf{\alpha}}^{SM}(\Sigma_{0,n}, G)} \omega_{\textrm{Goldman}}^{n-3} = \frac{(\pi\lambda)^{n-3}}{(n-3)!}~.\]

This ends the proof of Theorem \ref{t:SymplecticSuperMaximal}.

\section{Geometrization of super-maximal representations} \label{s:Geometrization}

In this section we prove Theorem \ref{t: geometrization}. We fix the numbers $\alpha_1,\ldots, \alpha_n\in (0,2\pi)$ such that
$$2\pi (n-1) < \sum _i \alpha_i < 2\pi n .$$ 

Recall the following uniformization theorem:

\begin{theo}[Troyanov, \cite{Troyanov}] \label{t: Troyanov}
Let $\Sigma$ be a compact Riemann surface of genus $g$, and $D = \sum_i \kappa_i p_i  $ be a divisor on $\Sigma$ with coefficients $\kappa_i \in (-\infty, 2\pi]$. Assume that its degree $ \kappa(D) = \sum _i \kappa_i$ satisfies 
$$ \kappa > 2\pi \chi(\Sigma). $$  
Then there exists a unique conformal metric $g_D$ on $\Sigma \setminus \text{Supp} (D)$ having curvature $-1$, and whose completion at each $p_i$ is either a cone of angle $\theta_i = 2\pi -\kappa_i$ if $\kappa_i < 2\pi $, or a parabolic cusp if $\kappa_i=2\pi$. Here $\text{Supp} (D)$ is the union of the $p_i$'s.
\end{theo}

Suppose that $\Sigma = \ProjC{1}$ is the Riemann sphere, and let $p_1,\ldots , p_n$  distinct points on $\ProjC{1}$. Assume now that $Q \in \text{Sym}^{n-3} (\ProjC{1})$, and let 
$$  D := \sum _{i= 1} ^n \alpha_i p_i -2\pi  Q .$$ 
 We have $\sum_i \alpha_i - 2\pi (n-3) > 2\pi \chi(\ProjC{1})$ so Troyanov's theorem yields a conformal metric $g_D$ on $\ProjC{1} \setminus \text{supp}(D)$.  Its completion at the $p_i$'s are cones of angle congruent to  $\overline{\alpha_i}$ modulo $2\pi\mathbb Z$ (depending if some $q_j$'s coalesce with $p_i$), whereas at the points of $\text{supp}(Q)$ they are cones of angle a multiple of $2\pi$ (can be a high multiple if some $q_j$'s coalesce). In particular, there exists a holomorphic map $f : \widetilde{\ProjC{1}\setminus \{ p_1,\ldots , p_n \}} \rightarrow \mathbb H^- $ such that  
\begin{equation} \label{eq: pull-back}     g_D = f^ * \big( \frac{dx^2 + dy^2 } {y^2} \big) \end{equation}
This map $f$ is unique up to post-composition by an orientation preserving isometry of $\mathbb H^-$, namely by an element of $G$. In particular, the map $f$ is equivariant with respect to a representation $\rho : \pi_1(\ProjC{1} \setminus \{ p_1,\ldots , p_n \}) \rightarrow G$, which is well-defined up to conjugacy by an element of $G$. This representation is called the holonomy of $g_D$.

\begin{lem}\label{l: parametrization}
$\rho$ belongs to $\Rep_\alpha ^{SM} (\pi_1(\ProjC{1} \setminus \{ p_1,\ldots , p_n \}), G)$.
\end{lem}

\begin{proof}
By construction $\theta_i (\rho ) = \alpha_i $. We get 
$$ \euler (\rho) = \frac{1}{2\pi} \big( -\text{Vol} (g_D) + \sum \alpha_i  \big) = n-1$$
by Gauss-Bonnet formula. 
\end{proof}

Let $\mathcal M_{0,n}$ and $\mathcal T_{0,n}$ respectively denote the moduli space and the Teichm\"uller space of $\Sigma_{0,n}$. Those spaces are complex manifolds that can be described in the following way: $\mathcal M_{0,n}$ is identified with the set of tuples $(p_1,\ldots, p_{n-3}) \in (\ProjC{1} \setminus \{ 0, 1, \infty \})  ^ {n-3}$ of distincts points, and $\mathcal T_{0,n}$ is the universal cover of $ \mathcal M_{0,n}$. In this description, the conformal structure corresponding to the tuple $(p_1,\ldots, p_{n-3})$ is $\ProjC{1} \setminus \{ p_1,\ldots, p_n \}$, where $p_{n-2}=0$, $p_{n-1}= 1$ and $p_n = \infty$. An element of $\mathcal T_{0,n}$ is the data of a tuple $(p_1,\ldots, p_{n-3})$ as before, plus a class modulo isotopy of diffeomorphisms $\phi : \Sigma_{0,n} \rightarrow \ProjC{1}\setminus \{p_1,\ldots, p_n\}$.

Lemma \ref{l: parametrization} enables to define a map 
$$  H : \mathcal T_{0,n} \times \text{Sym}^{n-3} (\ProjC{1}) \rightarrow \mathcal T_{0,n} \times \Rep^{SM}_\alpha (\pi_1(\Sigma_{0,n}), G),$$
by the formula
\begin{equation} H (p_1,\ldots, p_n, [\phi], Q ) =  (p_1,\ldots, p_n, [\rho \circ \phi_* ] ) \end{equation}
where $[\rho] $ is the holonomy of the metric $g_D$ on $\ProjC{1} \setminus \text{Supp} (D)$, with $D = \sum_i \alpha_i p_i - 2\pi Q$. As explained above, the holonomy is well defined on the fundamental group of $\ProjC{1} \setminus \{p_1,\ldots, p_n\}$.

\begin{lem}
$H$ is injective.
\end{lem}

\begin{proof}
 Assume that we are given disctinct points $p_1,\ldots, p_n$, and two elements $Q, Q' \in \text{Sym}^{n-3} (\mathbb C {\bf P}^1)$ such that 
$$H(p_1,\ldots, p_n, Q) = H(p_1,\ldots, p_n, Q') .$$ 
If $D= \sum \alpha_i p_i - 2\pi Q$ and $D'= \sum \alpha_i p_i - 2\pi Q$, we can find developing maps $f$ and $f'$ of $g_D$ and $g_{D'}$ respectively that are equivariant with respect to the same representation $\rho : \pi_1(\ProjC{1} \setminus \{ p_1,\ldots, p_n \})\rightarrow G$. 

Consider the following function 
$$ \delta = \varphi \circ d(f, f')$$ 
where $d(.,.)$ is the hyperbolic distance in $\mathbb H$, and where $\varphi =  \frac{\cosh - 1}{2} $. The function $\varphi$ descends to a function defined on $\ProjC{1} \setminus \{ p_1,\ldots, p_n\}$. Notice that by Schwarz lemma, $f$ and $f'$ are Lipschitz from the uniformizing metric of $\mathbb P^1 \setminus \{ p_1,\ldots, p_n \} $ to $\mathbb H$ so that $f$ and $f'$ have the same limit in each connected component of a preimage of a cusp: the fixed point of the $\rho$ image of the stabilizer of the component.
 In particular, the function $\varphi$ tends to $0$ at each of the cusps $p_i$'s. 

Assume by contradiction that $\delta$ is not identically zero. The set of zero values is then discrete by holomorphicity of $f-f'$.  Identifying holomorphically the upper half-plane $\mathbb H$ with the unit disc $\mathbb D$, we have 
$$ \log \delta  = \log |f - f' |^2 - \log (1- |f|^2) -  \log (1- |f'|^2) .$$
In particular $\log \delta$ is strictly subharmonic everywhere, appart eventually at the points where $f$ and $f'$ have zero derivative. This contradicts the maximum principle. 

The lemma follows.
\end{proof}

\begin{prop}
$H$ is continuous.
\end{prop}

\begin{proof}
Assume that  $$ (p_1^k, \ldots, p_n^k , [\phi^k ], Q^k) \rightarrow_{k\rightarrow \infty} (p_1,\ldots, p_n , [\phi], Q).$$ One can choose the diffeomorphisms $\phi_{k} : \Sigma_{0,n} \rightarrow \mathbb P^1 \setminus\{ p_1^k,\ldots, p_n^k\}$ in such a way that they converge uniformly on compact subsets of $\Sigma_{0,n}$ to $\varphi$ in the smooth topology.

The family of metrics $g_k := (\phi^k)^* g_{D_k}$ is bounded on each compact set of $\Sigma_{0,n}\setminus \phi^{-1} (\text{Supp}(Q))$, by the Schwarz Lemma. Since it satisfies an elliptic PDE, it is bounded in the smooth topology on compact sets. In particular, to prove our claim, it suffices to prove that if a subsequence $g_{k_j}$ converges in the smooth topology to some metric $g_\infty$ on compact subsets of $\Sigma_{0,n} \setminus \phi^{-1} (Q)$, then in fact $g_\infty = \varphi ^* g_D$. Indeed, this will prove that $g_k$ converges to $\varphi^* g_D$ in the smooth topology, and in particular, that the holonomies $\rho_k$ of $(\phi^k)^* g_{D_k}$ tend to the holonomy $\rho $ of $\phi^* g_D$ when $k$ tends to infinity. 

So let us assume in the sequel that $g_k$ converge to $g_\infty$ when $k$ tends to infinity, and let us prove that $g_\infty = \phi^* g_D$. By Schwarz lemma again, the metric $g_\infty$ is bounded by  $\phi^* g_P$, where $g_P$ is the Poincar\'e metric on $\ProjC{1} \setminus \big( \{ p_1, \ldots, p_n\} \cup \text{supp} (Q) \big)$. In particular,  at each point $q$ of $\phi^{-1} (\text{Supp} (Q) ) \cup \{1,\ldots,n\}$ it admits a completion isometric to a cone. Denote by $\kappa_\infty (q)$ the curvature of $g_\infty$  at such a point. By the unicity part of Troyanov uniformization theorem, it suffices to prove that $\kappa_\infty(q)$ is the coefficient of $q$ in the divisor $D$.

 Burger, Iozzi and Wienhard proved that the volume of a representation depends continuously on this latter, see \cite[Theorem 1]{BIW}. In particular, we have the following property 
\begin{equation}\label{eq: conservation of volume} \int _{\Sigma_{0,n}} \text{vol} (g_\infty) = \text{Vol} (\rho_\infty) = \lim _{k\rightarrow \infty} \text{Vol} (\rho_k) = \lim _{k\rightarrow \infty} \int _{\Sigma_{0,n}} \text{vol} (g_k), \end{equation}
where $ \text{vol} (g)$ stands for the volume form of $g$ on $\Sigma_{0,n}$.

Let $U\subset \overline{\Sigma_{0,n}}$ be any open set with smooth boundary, not containing points of $\phi^{-1} (\text{Supp}(D))$ on its boundary.  The uniform convergence of $g_k$ to $g$ on compact subset of $\Sigma_{0,n} \setminus \phi^{-1}(\text{Supp}(D))$, together with Fatou Lemma, show that 
\begin{equation}\label{eq: Fatou} \int _U \text{vol} (g_\infty) \leq \liminf _{k\rightarrow \infty} \int _U \text{vol} (g_k) . \end{equation}
Any strict inequality in \eqref{eq: Fatou} would contradict the conservation of volume \eqref{eq: conservation of volume}. So 
$$ \int _U \text{vol} (g_\infty) = \lim _{k\rightarrow \infty} \int _U \text{vol} (g_k) .$$
Applying Gauss Bonnet to the metrics $g_k$ on $U$, and using the uniform convergence of $g_k$ to $g_\infty$, we get 
$$ \sum _{q \in U\cap \phi^{-1} (\text{Supp} (D))} \kappa_\infty (q) = \lim_{k\rightarrow  \infty} \sum _{q \in U\cap \phi^{-1} (\text{Supp} (D))} \kappa_k (q). $$ 
This concludes the proof that $\kappa_\infty(q)$ is the coefficient of $D$ at the point $q$. 
\end{proof}

\begin{coro}
$H$ is a homeomorphism. In particular, for $\tau\in \mathcal T_{0,n}$, the map
$$H(\tau, \cdot): \text{Sym}^{n-3} (\mathbb C {\bf  P}^1)\simeq \mathbb C {\bf  P}^{n-3} \rightarrow \Rep _\alpha^{SM} (\pi_1(\Sigma_{0,n}), G) $$
is a homeomorphism.
\end{coro}

\begin{proof}
The map $H$ is a continuous, proper, injective map between connected manifolds of the same dimension. So the result is a consequence of the invariance of domain theorem.
\end{proof}

\bibliographystyle{alpha-fr} 
\bibliography{bibli} 

\end{document}